\documentclass{amsart}

\usepackage{amsmath,amsthm,amsfonts,amssymb,amscd}
\usepackage[all]{xy}
\usepackage{enumerate,verbatim,pdfsync}
\usepackage[english,francais]{babel}
\usepackage[dvips]{graphicx}
\usepackage[T1]{fontenc}
\usepackage[latin1]{inputenc}
\usepackage{ae}

{\theoremstyle{plain}
\newtheorem{theorem}{Theorem}[section]
\newtheorem{proposition}{Proposition}[section]
\newtheorem{lemma}{Lemma}[section]

\newtheorem{remark}{Remark}[section]
\newtheorem{corollary}{Corollary}[section]
}
\newcommand{\nc}{\newcommand}
\nc {\hh}{\check{h}}
\nc {\DD}{\mathbb{D}}
\nc {\CC}{\mathbb{C}}
\nc {\Pp}{\mathbb{P}}
\nc {\Ss}{\mathcal{S}}
\nc {\PP}{\mathbb{P}^{2}}
\nc {\Dx}{\frac{\partial}{\partial x}}
\nc {\Dy}{\frac{\partial}{\partial y}}
\nc {\Pd}{ \check{\mathbb{P}}^{2}}
\nc {\WW}{\mathcal{W}}
\nc {\Sym}{\mathrm{Sym}}
\nc {\Sing}{\mathrm{Sing}\,}
\nc {\Leg}{\mathrm{Leg}\,}
\nc {\Tang}{\mathrm{Tang}\,}
\nc {\OO}{\mathcal{O}}
\nc {\UU}{\mathcal{U}}
\nc {\EE}{\mathcal{E}}
\nc {\MM}{\mathcal{M}}
\nc {\KK}{\mathcal{K}}
\nc {\PW}{\mathcal{P}}
\nc {\NW}{\mathcal{N}_{\WW}}
\nc {\FF}{\mathcal{F}}
\nc {\GG}{\mathcal{G}}
\nc {\ZZ}{\mathcal{Z}}
\nc {\LL}{\mathcal{L}}
\nc {\HH}{\mathcal{H}}
\nc {\NN}{\mathcal{N}}
\nc {\VV}{\mathcal{V}}
\nc {\Ww}{\mathbb{W}}
\nc {\QQ}{\mathbb{Q}}
\nc {\II}{\mathcal{I}}

\date{}

\begin{document}

\title[Flat 3-webs of degree one]{Flat 3-webs of degree one on the projective plane}

\author{A. Beltrán$^1$}
\author{M. Falla Luza$^2$}
\author{D. Marín$^3$}

\thanks{1. Departamento de Matemática, Pontificia Universidad Católica del Perú, Av. Universitaria 1801, San Miguel, Lima, Perú, abeltra@pucp.edu.pe }

\thanks{2. Departamento de An\'alise -- IM, Universidade Federal Fluminense, M\'ario Santos Braga s/n -- Niter\'oi, 24.020-140 RJ Brasil,
 maycolfl@gmail.com}
             
\thanks{3. Departament de Matem\`{a}tiques, Universitat Aut\`{o}noma de Barcelona, E-08193  Bellaterra (Barcelona) Spain, davidmp@mat.uab.es} 
\thanks{The first author was partially supported by PUCP, Pontificia Universidad Católica del Perú. The second author was partially supported by PUCP-Perú and CAPES/Mathamsud. The first and third author were partially supported by FEDER / Ministerio de Econom\'{\i}a y Competitividad of Spain, grant MTM2011-26674-C02-01.}


\selectlanguage{francais}
\begin{abstract}
Le but de ce travail est d'étudier les $3$-tissus globaux ayant courbure nulle. En particular, nous nous intéressons aux feuilletages de degré $3$ dont le tissu dual est plat. L'ingrédient principal est la transformée de Legendre, qui est un avatar de la dualité projective classique dans le domaine des équations différentielles. Nous obtenons une characterization des feuilletages de degré $3$ sur le plan projectif dont les tissus duaux ont courbure nulle.
\end{abstract}

\selectlanguage{english}
\begin{abstract}
The aim of this work is to study global $3$-webs with vanishing curvature. We wish to investigate degree $3$ foliations for which their dual web is flat. The main ingredient is the Legendre transform, which is an avatar of classical projective
duality in the realm of differential equations. We find a characterization of degree $3$ foliations whose Legendre transform are webs with zero curvature.

\bigskip
\bigskip
\bigskip

\noindent {\tiny RÉSUMÉ}. \ 
Le but de ce travail est d'étudier les $3$-tissus globaux de courbure nulle. En particulier, nous nous intéressons aux feuilletages de degré $3$ dont le tissu dual est plat. L'ingrédient principal est la transformée de Legendre, qui est un avatar de la dualité projective classique dans le domaine des équations différentielles. Nous obtenons une caractérisation des feuilletages de degré $3$ sur le plan projectif dont les tissus duaux sont de courbure nulle.
\end{abstract}

\maketitle 


\section*{Introduction}

Roughly speaking, a web is a finite family of foliations. The study of web geometry has its birth in the late 1920's with Blaschke and his school. Recently, the study of holomorphic webs globally defined on compact complex manifolds started to be pursued, see for instance \cite{MarPer}, \cite{PiPer} and references therein. 

The main purpose of this paper is to study projective flat $3$-webs of degree one through projective duality. More concretely, we try to clarify and expand the notions and results studied in the paper \cite{MarPer} about degree one webs with the help of the so called Legendre transform (see Section 1.3). This duality associates to any $3$-web of degree one with a degree $3$ foliation in the dual plane.
 
The general philosophy behind our approach (considered for the first time in \cite{CDQL} as far as webs are concerned) is that flatness is characterized by the vanishing of a meromorphic curvature $2$-form having poles along the discriminant of the $3$-web.
In \cite{PiPer,MarPer} it is shown that generic invariant components of the discriminant do not produce poles in the curvature. Following this principe we  decompose the discriminant in two parts: a roughly transverse part coming from the inflection curve of the dual foliation and an invariant one due to the singularities. The last part not being decisive for the flatness, see Theorem~\ref{teorema-principal}.

In the first chapter we introduce our notation and main objects such as the curvature, discriminant, Gauss map and Legendre transform. In Chapter 2 we investigate the relationship between the geometry of a degree $3$ foliation and its dual web, more specifically, we study the geometrical relation between the inflection divisor and singular points of the foliation with the discriminant of the associated web. In Chapter 3 we restrict ourselves to homogeneous foliations of degree $3$ and prove that, for these foliations, only the inflection divisor matters for the flatness of the dual web. In particular we get a classification of generic homogeneous foliations of degree $3$ having flat Legendre transform. Finally, in last chapter we state our main result, showing that if the curvature has no pole on the curve associated to the inflection of the dual foliation, then it vanishes on $\PP$. Finally, using some notions introduced in \cite{MarPer} we  give a complete geometric  characterization of degree $3$ saturated foliations with flat dual web. As an application we obtain that the dual webs of convex foliations are always flat, generalizing a previous result in \cite{MarPer}.


Since we do not want to exclude artificially reducible $3$-webs of degree one from our study, we must also consider the case of foliations with reduced $1$-dimensional singular set. Using the techniques developed in this paper we hope to obtain a fully classification of reducible degree $1$ flat $3$-webs in the near future.

\section{Global foliations, webs and Legendre transform}
\subsection{Local definitions}

A foliation $\FF$ on a complex surface $S$ is locally given by 
an open covering $\{U_{i}\}_{i\in I}$ of $S$ and
holomorphic $1$-forms $\omega_i$ in $U_i$, $i\in I$, such that for each $i,j\in I$ with
  $U_i \cap U_j \neq \emptyset$, we have the relation
$$
\omega_i = g_{ij} \cdot \omega_j
$$
for some non-vanishing function $g_{ij} \in \OO^{*}(U_i \cap U_j)$. The foliation $\FF$ is said to be \textbf{saturated} if the codimension of the zero set of the $1$-forms $\omega_i$ is  $\ge 2$. For a thorough treatment we refer the reader to \cite{Br}.

We will restrict ourselves to the case $S=\Pp^2$. In this case a foliation is defined by a polynomial $1$-form  $a(x,y) dx + b(x,y) dy$ on $\mathbb C^2$. In the same way,  a $k$-web $\WW$ on the projective plane is defined by a $k$-symmetric polynomial $1$-form
\[
\omega = \sum_{i+j = k} a_{ij}(x,y) dx^i dy^j
\]
with non identically zero discriminant. Here, we refer to the discriminant $\Delta(\WW)$ as the set of points where $\omega$ does not factor as the product
of $k$ pairwise linearly independent 1-forms. In fact, $\Delta(\WW)$ can be endowed with a natural divisor structure, see for instance \cite[\S1.3.4]{PiPer}, but we do not use it.
 In more intrinsic terms, a $k$-web on $\Pp^2$ is defined by an element $\omega$ of $H^0(S,\Sym^k \Omega^1_{\Pp^2}\otimes N)$ for a suitable line bundle $N$, still
subjected to the condition above: non-zero discriminant.
It is natural to write $N$ as $\mathcal O_{\mathbb P^2}(d + 2k)$
since the pull-back of $\omega$ to a  line $\ell \subset \mathbb P^2$ will
be a global section of $\Sym^k \Omega^1_{\mathbb P^1} ( d+ 2k) = \mathcal O_{\mathbb P^1}(d)$ and consequently for a generic $\ell$ the integer  $d$, called the \textbf{degree} of the web,  will count the number of tangencies between $\ell$ and the leaves of the $k$-web $\mathcal W$ defined by $\omega$. That said, we promptly
see that the space of $k$-webs on $\mathbb P^2$ of degree $d$ is an open subset of $\mathbb P H^0(\mathbb P^2, \Sym^k \Omega^1_{\mathbb P^2}(d+2k))$, we denote it briefly by $\Ww(k,d)$.

\subsection{Flat webs}\label{1.2}

One of the first results of web geometry, due to Blaschke-Dubour\-dieu,  characterizes the local equivalence of a germ of  $3$-web $\mathcal W$ on $\mathbb C^2$ with the trivial $3$-web defined by $dx \cdot dy \cdot (dx-dy)$ through the vanishing of a differential covariant: the curvature of $\mathcal W$. 
%
%
It is a holomorphic $2$-form $K(\WW)$ which can be easily defined locally in the complement of the discriminant of the $3$-web $\WW$ and satisfies the covariant relation $\varphi^* K ( \mathcal W) = K ( \varphi^* \mathcal W)$ for any holomorphic submersion $\varphi$. This property allows to see that it extends meromorphically to $\Delta(\WW)$. For more details we refer to \cite[\S2.2]{MarPer}.

For a $k$-web $\mathcal W$ with $k>3$, one defines the curvature of $\mathcal W$ as the sum of the curvatures of  all $3$-subwebs of $\mathcal W$. It is again a differential covariant, and to the best of our knowledge there is no result characterizing its  vanishing.  Nevertheless,  according to a result of Mihaileanu, this vanishing is a necessary condition for the maximality of the rank of the web, see \cite{Henaut,Ripoll}  for a full discussion and   pertinent references. 

The $k$-webs with zero curvature are called {\bf flat $k$-webs}. Global flat $k$-webs on $\PP$ of degree $d$ form a Zariski closed subset of $\Ww(k,d)$ and it is our purpose to study flat $3$-webs of degree one through geometrical conditions on the Legendre transform, which we discuss in the next subsection.

Since the curvature of a web $\mathcal W$ on a complex surface is a meromorphic $2$-form with poles contained in the discriminant $\Delta(\mathcal W)$ of
$\mathcal W$ and there are no  holomorphic $2$-forms on
$\mathbb P^2$,
 the curvature of a global  web  $\mathcal W $ on the projective plane is zero if and only if it is holomorphic at the generic points of the irreducible components of $\Delta(\mathcal W)$. In the sequel we will apply the following result which is the particularization of \cite[Theorem~1]{MarPer} to the case $k=3$:
 
\begin{theorem}\label{baricentro}
Let $\WW$ be a $3$-web on $(\CC^{2},0)$ with smooth (but not necessarily reduced), and non empty discriminant. Assume that $\WW$ is the superposition $\WW_{2}\boxtimes\FF$ of a foliation $\FF$ and a $2$-web $\WW_{2}$ satisfying $\Delta(\WW_{2})=\Delta(\WW)$. Then the curvature of $\WW$ is holomorphic along $\Delta(\WW)$ if and only if $\Delta(\WW)$ is invariant by either $\WW_{2}$ or $\FF$.
\end{theorem}

In fact, in \cite{CDQL} the authors introduce the notion of $\FF$-barycenter $\beta_{\FF}(\WW)$ of a completely decomposable web $\WW$ with respect to a foliation $\FF$, which was extended in \cite{MarPer} to include the case $\FF$ is a web. The general statement of Theorem 1 in \cite{MarPer} is that the holomorphy of $K(\WW_{2}\boxtimes\WW')$ is equivalent to the fact that $\Delta(\WW_{2})$ be invariant by either $\WW_{2}$ or $\beta_{\WW_{2}}(\WW')$. In our situation $\WW'=\FF$ is a foliation and the invariance with respect to the $2$-web $\beta_{\WW_{2}}(\FF)=\FF\boxtimes\FF$ is nothing but the invariance with respect to $\FF$.

\medskip

\subsection{The Legendre transform}

From now we will focus on global $k$-webs of degree $d$ on the projective plane. As it was already mentioned these are determined by a global section $ \omega$ of $\Sym^k \Omega^1_{\mathbb P^2}(d+ 2k)$ having non-zero discriminant. Also, a $k$-web of degree $d$ can be expressed as a global section $X$ of $\Sym^k T\mathbb P^2( d-k)$ subjected to the same condition as above. 
%
%
From Euler's sequence
\[
\begin{array}{ccccccccc}
0 & \to & \mathcal O_{\mathbb P^2} & \longrightarrow  & \mathcal O_{\mathbb P^2}(1)^{\oplus 3} & \longrightarrow  & T{\mathbb P^2}  & \to 0\\
 & & \lambda & \mapsto & (\lambda x,\lambda y,\lambda z) & & & &\\
 & &   &         & (A,B,C) & \mapsto & A\partial_{x}+B\partial_{y}+C\partial_{z} & &
\end{array}
\]
 we can deduce the following exact sequence
\[
0 \to \Sym^{k-1} ( \mathcal O_{\PP} (1)^{\oplus 3} ) \otimes \mathcal O_{\PP} \to  \Sym^{k} ( \mathcal O_{\PP} (1)^{\oplus 3} ) \to \Sym^k T\PP \to 0 \, .
\]
After tensorizing with the line bundle $\mathcal O_{\PP}(d-k)$,
it implies that every global section $X$ of $\Sym^{k}T\PP(d-k)$ is obtained from 
a global section of $\Sym^{k}(\mathcal O_{\PP}(1)^{\oplus 3})\otimes\mathcal O_{\PP}(d-k)$, i.e. a bihomogeneous polynomial $P(x,y,z;a,b,c)$ of degree $d$  in the coordinates $(x,y,z)$ and degree  $k$ in the coordinates $(a,b,c)$, by replacing in it the variables $a,b,c$ by $\partial_{x},\partial_{y},\partial_{z}$ respectively.
Moreover,  two bihomogeneous polynomials $P$ and $P'$ determine the same global section $X$ if and only if they differ by a (necessarily bihomogeneous) multiple of $xa + yb + zc$.


By using homogeneous coordinates  in the dual projective plane $\check{\Pp}^{2}$ which associates to the point
$(a:b:c)\in\check{\Pp}^{2}$ the line $\{  ax+by+cz=0\} \subset \PP$ one can identify the cotangent space of $\PP$ at the point of homogeneous coordinates $(x:y:z)$ with 
\begin{eqnarray*}
T^*_{(x:y:z)}\PP&=&\{\omega=a\,dx+b\,dy+c\,dz\in T^*\mathbb{C}^{3}: \omega(R)=0\}\\ &=&\{a\,dx+b\,dy+c\,dz : ax+by+cz=0\}
\end{eqnarray*}
with the incidence variety
$$\mathcal{I}=\{((x:y:z),(a:b:c)) | ax+by+cz=0\}\subset \PP\times\check{\Pp}^{2}.$$

Let $\mathcal{W}$ be a $k$-web of degree $d$ on $\PP$ defined by the bihomogeneous polynomial $P(x,y,z;a,b,c)$. Then $S_{\mathcal{W}}\subset\Pp T^{*}\PP$, the graph of $\mathcal W$ on $\Pp T^* \PP$,   is given by
$$S_{\mathcal{W}}=\{((x:y:z),(a:b:c))\in\PP\times\check{\Pp}^{2} | ax+by+cz=0, P(x,y,z;a,b,c)=0\}$$
under the above identification between $\mathcal{I}$ and $\Pp T^{*}\PP$.

Suppose $\mathcal W$ is an irreducible web of degree $d>0$ and consider the restrictions  $\pi$ and $\check{\pi}$ to $S_{\mathcal{W}}$ of the natural projections of $\PP\times\check{\Pp}^{2}$ onto $\PP$ and $\check{\Pp}^{2}$ respectively. These projections $\pi$ and $\check{\pi}$ are rational maps of  degrees $k$ and $d$ respectively. The contact distribution $\mathcal{D}$ on $\Pp T^{*}\PP$ given by
$$\mathcal{D}=\ker(a\,dx+b\,dy+c\,dz)=\ker(x\,da+y\,db+z\,dc)$$
defines a foliation $\FF_{\mathcal{W}}$ on $S_{\mathcal{W}}$ which projects through $\pi$ onto the $k$-web $\mathcal{W}$ and by $\check{\pi}$ onto a $d$-web $\check{\mathcal{W}}$ on $\check{\Pp}^{2}$. The $d$-web $\check{\mathcal{W}}$ is called the \textbf{Legendre transform} of $\mathcal{W}$ and it  will be denoted by $\Leg(\mathcal{W})$.

Browsing classical books on ordinary differential equations one can find the Legendre transform as an involutive transformation between polynomial differential equations, see for instance \cite[Part I Chapter II Section 2.5]{Ince}.

Of course we can proceed to define the Legendre transform for arbitrary $k$-webs of  arbitrary degree $d$. Notice that when $\mathcal W$ decomposes as the product of two webs $\mathcal W_1 \boxtimes \mathcal W_2$ then its Legendre transform will be the product of $\Leg(\mathcal W_1)$ with $\Leg(\mathcal W_2)$.

\begin{remark}
Let us fix a generic line $\ell$ on $\PP$ and consider  the tangency locus $\Tang(\mathcal{W},\ell)=\{p_{1},\ldots,p_{d}\}\subset\PP$ between $\WW$ and $\ell$.  We can think $\ell$ as a point of $\check{\Pp}^{2}$ and the dual $\check{p}_{i}$ as straight lines on $\check{\Pp}^{2}$ passing through the point $\ell$. Then the set of tangent lines of $\Leg\WW$ at $\ell$ is just $T_{\ell}\Leg\mathcal{W}=\bigcup\limits_{i=1}^{d}\check{p}_{i}$.
\end{remark}

Consider affine coordinates $(x,y)$ of an affine chart of $\PP$ and an affine chart of $\check{\Pp}^{2}$ whose affine coordinates $(p,q)$ correspond to the line $\{y=px+q\}\subset\PP$. If a web $\WW$ is defined by an implicit affine equation $F(x,y;p)=0$ with $p=\frac{dy}{dx}$ then $\Leg(\WW)$ is defined by the implicit affine equation
$$
\check{F}(p,q;x):=F(x,px+q;p)=0,\qquad\textrm{with}\qquad x=-\frac{dq}{dp}.
$$
In particular, for a foliation  defined by a vector field $A(x,y)\frac{\partial}{\partial x}+B(x,y)\frac{\partial}{\partial y}$  we can take $F(x,y;p)=B(x,y) - pA(x,y)$ and therefore its Legendre transform is given by
\begin{equation}\label{F}
\check{F}(p,q;x)= B(x,px+q) - pA(x,px+q).
\end{equation}
At this point we recall briefly the definition of the \textbf{Gauss map} of a holomorphic foliation $\FF$. When the foliation is saturated, the Gauss map of $\FF$ is the rational map $\GG_{\FF}: \PP \dashrightarrow \Pd$, $\GG_{\FF}(p)=T_{p}\FF$, which is well defined outside the singular set $\Sing(\FF)$ of $\FF$. If the foliation is given by the homogeneous form $\omega = a(x,y,z)dx + b(x,y,z) dy + c(x,y,z) dz$, then the Gauss map can be written as 
$$
\GG_{\FF}(p)=[a(p): b(p): c(p)].
$$
This allows us to define the Gauss map of a non saturated foliation $\FF$ as $\GG_{\FF}:= \GG_{\FF'}$, where $\FF'$ is the saturation of $\FF$.
Moreover, for the purposes of this paper it will be useful to define $\GG_{\FF}(C)$ as  the closure of $\GG_{\FF}(C\setminus\Sing(\FF'))$ if $C$ is a (not necessarily irreducible) curve on $\PP$ passing through some singular points of the saturation $\FF'$ of $\FF$. Notice that $\GG_{\FF}(C)$ is defined merely as a set, i.e. no divisor structure is considered over it in this paper.

\section{Inflection divisor and singularities versus discriminant}

Let $\FF$ be a degree $d$ foliation on $\PP$ given by the degree $d$ homogeneous vector field $X$. The \textbf{inflection divisor} of $\FF$, denoted by $I(\FF)$, is the divisor defined by the vanishing of the discriminant determinant
\begin{equation}\label{inflection}
\det \left( \begin{array}{ccc}
  x & y & z \\
  X(x) & X(y)  & X(z) \\
  X^2(x) & X^2(y)  & X^2(z)
\end{array} \right) \, .
\end{equation}

This divisor has been studied in \cite{Per} in a more general context. In particular, the following properties were proven.
\begin{itemize}
\item[(a)] If the determinant (\ref{inflection}) is identically zero then $\FF$ admits a
rational first integral of degree $1$; that is,  if we suppose that
the singular set of $\FF$ has codimension $2$ then the degree of
$\FF$ is zero;
\item[(b)]  $I(\FF)$ does not depend on a particular choice of a system of homogeneous coordinates $(x : y : z)$ on $\PP$ and it coincides on $\PP \setminus \Sing({\FF})$ with the curve described by the inflection points of the leaves of $\FF$;
\item[(c)] If $C$ is an irreducible algebraic invariant curve of $\FF$ then
$C \subset I(\FF)$ if, and only if, $C$ is an invariant line;
\item[(d)] The degree of $I(\FF)$ is exactly $3d$.
\end{itemize}

\begin{remark} Let $\FF'$ be a saturated foliation on $\PP$ given by a homogeneous vector field $X'$ and let $P$ be a homogeneous polynomial.
It follows from the definition that the inflection divisor of the non saturated foliation $\FF$ given by the vector field $X=P \cdotp X'$ satisfies
$$
I(\FF)= I(\FF') + 3 C
$$
where 
$C$ is de divisor given by $P$.
\end{remark}

As mentioned above, the curvature of a web $\WW$ on a complex surface is a meromorphic 2-form with poles contained in the discriminant set $\Delta(W)$ of $\WW$. For dual of foliations, one has the following lemma.

\begin{lemma}\label{discriminant}
Let $\FF$ be a (not necessarily saturated) foliation on $\PP$. Then
$$\Delta (\Leg \FF) = \GG_{\FF}(I(\FF))\cup \check{\Sigma}(\FF)
$$
where $\check{\Sigma}(\FF)$ consists in the dual lines of the special singularities
$$\Sigma(\FF)=\{s \in \Sing(\FF')\,:\, \nu(\FF',s) \geq 2 \text{ or s is a radial singularity of $\FF'$} \}$$ of  the saturation $\FF'$ of $\FF$ and $\nu(\FF',s)$ stands for the algebraic multiplicity of $\FF'$ at~$s$, i.e. the maximum of the vanishing orders at $s'$ of the components of a local saturated vector field defining $\FF'$.
\end{lemma}

\begin{proof}
Let us suppose first that $\FF$ is a saturated foliation. We know  from \cite[\S3.2]{MarPer} that $\Delta(\Leg \FF)$ is formed by the image of the components of the inflection divisor by the Gauss map and the dual of some singularities.
Define $\tau(\mathcal F,s)$ as the infimum of the integers $k\ge\nu(\mathcal F,s)$ such that the $k$-jet at $s$ of a vector field $X$ defining $\mathcal F$ is not a multiple of the radial vector field centered at $s$. 
Take affine coordinates $(x,y)$ in $\PP$ such that $s=(0,0)$ and decompose the vector field $X=X_{\nu}+\cdots+X_{d+1}$ with $X_{i}=A_{i}\partial_{x}+B_{i}\partial$ and $A_{i},B_{i}$  homogeneous polynomials of degree $i$. Writing $\nu=\nu(\FF,s)$, $\tau=\tau(\FF,s)$ and $R=x\partial_{x}+y\partial_{y}$ we have $X_{i}=C_{i}R$ for $i=\nu,\ldots,\tau-1,d+1$ and $xB_{\tau}-yA_{\tau}\not\equiv 0$. Take $(p,q)$ affine coordinates in $\check{\PP}$ corresponding to the line $\{y=px+q\}\subset\PP$. The dual line $\check{s}$ of $s$ has then equation $q=0$. Evaluating the defining equation~(\ref{F}) of $\Leg\FF$ along $q=0$ we obtain that
\begin{equation}\label{F2}\check{F}(p,0,x)=x^{\tau}[B_{\tau}(1,p)-pA_{\tau}(1,p)]+\mathrm{higher\  order\ terms\ in\ }x.
\end{equation}
Let  $\ell\subset\mathbb{P}^{2}$ be a generic choice of a line passing through $s$, i.e. having dual coordinates $(p,0)$ with $p$ generic. Let $\ell'$ be a generic line of $\mathbb{P}^{2}$ close to $\ell$, i.e. having dual coordinates $(p',q')$ close to $(p,0)$. 
Formula~(\ref{F2}) shows that when $\ell'$ approaches to $\ell$ exactly $\tau(\mathcal F,s)$ of the $\deg\mathcal F$ tangency points between $\ell'$ and $\mathcal F$ collapse at $s\in\ell$. Since the dual lines of the tangency points define the directions of the dual web $\mathrm{Leg}\mathcal F$, now it is clear that the dual line $\check{s}$ of $s$ is contained in the discriminant of $\mathrm{Leg}\mathcal F$ if $\tau(\mathcal F,s)\ge 2$. This occurs when $\nu(\mathcal F,s)\ge 2$ or $s$ is a radial singularity of $\mathcal F$.
On the other hand, if $s$ is a non radial singularity of multiplicity one and $\check{s}\subset\Delta(\mathrm{Leg}\mathcal F)$ then for each line $\ell$ passing through $s$ there is a multiple tangency point between $\ell$ and $\mathcal F$ which splits in simple tangency points between  $\mathcal F$ and any generic line $\ell'$ close to $\ell$.  Since $\mathrm{Sing}\mathcal F$ is finite the only possibility for this collapse of tangencies is that all the tangent lines of some component of $I(\mathcal F)$ pass through the singular point $s$, i.e. $\check{s}$ being the image by the Gauss map of a component of $I(\mathcal F)$.
%

Consider now a non saturated foliation $\FF$ defined by the vector field $X= P\cdotp X'$ such that the foliation $\FF'$ given by $X'$ is saturated and denote by $D=\{P=0\}$. It is easily seen that $\GG_{\FF'}(D) = \Tang(\Leg D, \Leg \FF')$, so we have
$$\Delta(\Leg \FF)= \Delta(\Leg\FF')\cup \Tang(\Leg D, \Leg \FF')= \Delta(\Leg\FF') \cup \GG_{\FF'}(D).
$$
Now we just use the last remark to conclude our claim.
\end{proof}

Our main objective is to understand the degree $3$ foliations $\FF$ whose Legendre transform $\Leg\FF$ is a flat $3$-web. As we have pointed out in section~\ref{1.2}, flatness is equivalent to the holomorphy of the curvature at the generic point of each irreducible component of the discriminant. The main idea we want to exploit is that the important part of $\Delta(\Leg\FF)=\GG_{\FF}(I(\FF))\cup\check{\Sigma}(\FF)$
comes from the the image of the inflection curve under the Gauss map. In fact, we shall prove that the role of the dual lines of the special singularities is secondary in the sense that if the curvature is holomorphic along $\mathcal{G_{F}}(I(\mathcal F))$ then it is also holomorphic along $\check{\Sigma}(\FF)$ and consequently $\Leg\FF$ is flat.

We begin by studying the curvature along the dual line of a singular point of a foliation $\FF$ of degree $3$ in some particular situations.  There is no loss of generality in assuming that in affine coordinates the origin $O=(0,0)$ is a singularity of $\FF$. In this coordinates we can write a vector field defining $\FF$ as $X=X_1 + X_2 + X_3 + H \cdotp R$, where $X_i$ denotes a homogeneous vector field of degree $i$, $H$ is a homogeneous polynomial of degree $3$ and $R$ is the radial vector field. 

\begin{proposition}\label{holomorfia_singularidad_menor_2}
If $\nu(\FF, O) \leq 2$ then the curvature $K(\Leg \FF)$ is holomorphic at the generic point of $\check{O}$.
\end{proposition}

\begin{proof}
With the previous notation, let us assume first that $X_1 \neq 0$ and $\check{O} \subseteq \Delta(\Leg \FF)$. If $O$ is not a radial singularity, then the tangency order between $\FF$ and a generic line through $O$ is one. Formula~(\ref{F2}) implies that for each generic line $\ell$ passing through $O$ we have $\mathrm{Tang}(\FF,\ell)=O+P_{\ell}+Q_{\ell}$ with $P_{\ell}\neq O\neq Q_{\ell}$ and allow us to write, around a generic point of $\check{O}$, $\Leg(\FF)= \WW_1 \boxtimes \WW_2$ where $\WW_1$ is a foliation tangent to $\check{O}$ and $\WW_2$ is a $2$-web transversal to $\check{O}$ with $\Delta(\WW_2)=\check{O}$. Then we can apply Theorem~\ref{baricentro} to conclude the proof. In the case of a radial singularity we just use first Proposition~3.3 then Proposition 2.6 of \cite{MarPer} to get the holomorphy of $K(\Leg \FF)$ along $\check{O}$.

Suppose now $X_1=0$ and $X_2 \neq 0$. If $X_2$ is not parallel to the radial vector field $R$ then Formula~(\ref{F2}) implies that for each generic line $\ell$ passing through $O$ we have $\mathrm{Tang}(\FF,\ell)=2O+P_{\ell}$ with $P_{\ell}\neq O$.
Thus, around a generic point of $\check{O}$  one can write $\Leg \FF = \WW_1 \boxtimes \WW_2$, where $\WW_1$ is a foliation transverse to $\check{O}$ and $\WW_2$ is a $2$-web having $\check{O}$ as a totally invariant discriminant. Again, we can use Theorem~\ref{baricentro} to conclude that $K(\Leg \FF)$ is holomorphic at the generic point of $\check{O}$. Finally we consider the case $X_2=(\alpha x +\beta y)\cdotp R$ for some complex numbers $\alpha$ and $\beta$. Furthermore, we write
\begin{eqnarray*}
X_3 &= &(\sum_{i+j=3}a_{ij}x^i y^j)\frac{\partial}{\partial x} + (\sum_{i+j=3}b_{ij}x^i y^j)\frac{\partial}{\partial y}\\
H &= & h_0 x^3 + h_1 x^2 y + h_2 x y^2 + h_3 y^3.
\end{eqnarray*}
Then in the affine coordinates $(p,q)$ of $\Pd$ the web $\Leg \FF$ is given by the symmetric form $$\omega = c_3 dq^3 + q c_2 dq^2 \cdotp dp + q c_1 dq \cdotp dp^2 + q^2 c_0 dp^3,$$
where 
\begin{eqnarray*}
c_3 & = & -a_{0,3}p^4 + h_3 p^3 q + (b_{0,3}-a_{1,2})p^3 + h_2 p^2 q + (b_{1,2}- a_{2,1})p^2 + h_1 pq +\\ &&+(b_{2,1}-a_{3,0})p + h_0 q + b_{3,0}\\
c_2 & = & 3a_{0,3}p^3 - 3 h_3 p^2 q + (2 a_{1,2}- 3 b_{0,3})p^2 - 2h_2 pq + (a_{2,1} - 2b_{1,2})p - h_1 q - b_{2,1}\\
c_1 & = & -3 a_{0,3}p^2 q + 3h_3 pq^2 + (3b_{0,3}- a_{1,2})pq + h_2 q^2 + b_{1,2}q + \beta p + \alpha\\
c_0 & = & a_{0,3}pq - b_{0,3}q - h_3 q^2 - \beta.
\end{eqnarray*}
If $q | c_3(p,q)$ we deduce that $X= P_3 \cdotp R$ for some degree $3$ polynomial $P_3$. In this case, the $3$-web $\Leg\FF$ is algebraic hence $K(\Leg \FF) \equiv 0$ according to Mihaileanu criterion. Then we shall assume $q\nmid c_3(p,q)$ and consider the branched covering $\pi (p,s)=(p, s^2)=(p,q)$. It is a straightforward computation to see that the web $\pi^{*}(\Leg \FF)$ is regular at the generic point of $\{s=0\}$. We  write locally $K(\Leg \FF) = \frac{f(p,q)}{q^{\delta}} dp \wedge dq$ around a generic point of $\check{O}=\{q=0\}$, with $\delta\in\mathbb Z$ and $f(p,q)$ a holomorphic germ satisfying $q \nmid f(p,q)$.
Noting that
$$
K(\pi^{*}(\Leg \FF))= \frac{2f(p,s^2)}{s^{2\delta -1}}dp \wedge ds.
$$
we obtain $\delta \leq 1/2$  from what we conclude that $\frac{f(p,q)}{q^{\delta}}$ hence $K(\Leg \FF)$ is holomorphic at the considered (generic) point of
 $\check{O}$.
\end{proof}

As a consequence, in order to study the flatness of $\Leg \FF$ we just need to consider singularities of the form $X= X_3 + H\cdotp R$.

\begin{remark}
Let $\FF$ be the foliation given by a vector field $X= A \Dx + B \Dy$ around a singularity $p$. Then $\mu(\FF,p) \geq \nu(A,p)\cdotp \nu(B,p)$, where $\mu$ stands for the Milnor number and $\nu$ the algebraic multiplicity (see for instance \cite[\S3.5]{Fish}).
\end{remark}

\begin{lemma}\label{solo_una_singularidad_3}
Let $\FF$ be a foliation of degree $3$ on $\PP$. If $p \in \Sing(\FF)$ is such that $\nu(\FF,p) \geq 3$ then for any $s \in \Sing(\FF)\setminus\{p\}$ one has $\nu(\FF,s)\leq 2$.
\end{lemma}

\begin{proof}
If $\FF$ is a saturated foliation, the lemma is a consequence of the last remark and the Darboux theorem $\sum\limits_{s\in \Sing(\FF)} \mu(\FF,s)=3^2 + 3 + 1 = 13$ (see \cite{Br}).

Let us suppose that $\FF$ is non saturated and given by $X= P\cdotp X'$ such that the foliation $\FF'$ given by $X'$ is saturated and denote by $D=\{P=0\}$. We consider the following possibilities.
\begin{enumerate}[(i)]
\item If $\deg(D)=1$, then $\nu(\FF',p) \geq 2$ and by Darboux theorem we have, for any other singularity $s$ of $\FF'$, $\mu(\FF',s)\leq 3$. Therefore $\nu(\FF',s)\leq 1$ and so $\nu(\FF,s) = \nu(\FF',s) + \nu(D,s) \leq 2$.

\item If $\deg(D)=2$, then by Darboux theorem $\nu(\FF',s)
\leq 1$ for all $s \in \Sing (\FF')$. Since $3 \leq \nu(\FF,p)$ we infer $\nu(D,p)\geq 2$ and this implies that $D$ represents two different lines through $p$, therefore $\nu(D,s) \leq 1$ for any other singularity of $\FF'$.

\item If $\deg(D)=3$, assume first $p \in \Sing(\FF')$. Then $\nu(D,p)\geq 2$, therefore $\nu(D,s)\leq 2$ for any $s \neq p$, otherwise we would have that $D$ is the union of three different lines through $s$ and $\nu(D,p)\leq 1$. If $p \notin \Sing(\FF')$ then $\nu(D,p)=3$ and $D$ represents the union of three different lines through $p$. It remains to observe that $\nu(D,s)\leq 1$ for any $s \neq p$. 
\end{enumerate} 
\end{proof}

\section{The homogeneous case}
In this section we study the dual web of a degree $3$ homogeneous foliation on $\PP$. We recall that a homogeneous foliation of degree $d$ is, by definition, induced by a vector field $X_{d}$ with degree $d$ homogeneous polynomial coefficients in some affine coordinates. It can be checked that
 a degree $d$ foliation is homogeneous if and only if it has a singularity with algebraic multiplicity exactly $d$. Let us begin with a general result.
\begin{lemma}\label{holomorfia_homogeneo}
Let $\FF$ be a homogeneous foliation of degree $d$ on $\PP$ such that $K(\Leg \FF)$ is holomorphic on $\Pd \setminus \check{O}$. Then $K(\Leg \FF)\equiv 0$.
\end{lemma}

\begin{proof}
Fix coordinates $(a,b)$ in $\Pd$ associated to the line $\{ax+by=1\}$ in $\PP$. Since $\FF$ is invariant by the homoteties $h_{\lambda}(x,y)=(\lambda x,\lambda y)$ we can assert that $\Leg \FF$ is invariant by the dual maps $\check{h}_{\lambda}(a,b)=(\frac{a}{\lambda},\frac{b}{\lambda})$ and so
$$
\check{h}_{\lambda}^{*}(K(\Leg \FF))=K(\Leg \FF).
$$
Combining our hypothesis and the fact that $\check{O}$ is the line at infinite in coordinates $(a,b)$ yields $K(\Leg \FF)= f(a,b) da \wedge db$ where $f$ is a polynomial. By the previous discussion one has $\lambda^2 f(a,b)= f(\frac{a}{\lambda}, \frac{b}{\lambda})$, the assertion follows.
\end{proof}

Now we state the main result of this section.

\begin{proposition}\label{homogeneo_flat}
Let $\FF$ be a homogeneous foliation of  degree $3$ on $\PP$. Suppose that $K(\Leg \FF)$ is holomorphic at the generic point of $\GG_{\FF}(I(\FF))$, then $K(\Leg \FF) \equiv 0$.
\end{proposition}

\begin{proof}
Since $\nu(\FF,O)=3$,  Lemma~\ref{solo_una_singularidad_3} and Proposition~\ref{holomorfia_singularidad_menor_2} yields $K(\Leg \FF)$ to be holomorphic over $\PP \setminus \check{O}$. We need only to apply lemma \ref{holomorfia_homogeneo}. 
\end{proof}

\subsection{The generic case}

Consider now homogeneous foliations of degree $3$ such that $I(\FF)$ is reduced and suppose that $K(\Leg \FF)$ is holomorphic at the generic point of $\GG_{\FF}(I(\FF))$. Since $\FF$ is given by a homogeneous vector field $X=A \Dx + B \Dy$ then $I(\FF)$ is formed by the  line at infinity $\ell_{\infty}$ and other eight lines through $O=(0,0)$. The tangent cone of $\FF$ is the divisor $$\Tang(\FF,R)=\{yA-xB=0\}$$ and it consists of the invariant lines of $\FF$ different from $\ell_{\infty}$.

\begin{lemma}
Under the above assumptions, the tangent cone of $\FF$ is reduced.
\end{lemma}

\begin{proof}
Suppose the assertion is false. Then we would have some line $\ell$ with $\ell^2 \mid yA - xB$. It is a straightforward computation to show that $\ell^2 \mid I(\FF)$, which is a contradiction.
\end{proof}

By the previous lemma we can consider $I(\FF)= \ell_1+ \cdots + \ell_8+\ell_{\infty} $ with $\ell_1, \ldots, \ell_4$ being the lines through $O$ invariant by $\FF$ and $\ell_5, \ldots, \ell_8$ being transversal inflection lines.
Now we make some remarks.
\begin{enumerate}
\item Since $X$ is homogeneous, the tangent directions of $\FF$ restricted to any straight line through the origin are all parallel. Hence
for every $i \in \{5,6,7,8\}$ the restriction $T \FF \mid_{\ell_i}$ is formed by parallel lines through a point $p_i \in \ell_{\infty}$. Since $K(\Leg \FF)$ is holomorphic over $\GG_{\FF}(\ell_i)$ it follows from Theorem~\ref{baricentro} that $p_i \in \Sing(\FF)$.

\item Since $\deg(\FF)=3$ and $\FF$ is saturated we have four different points $p_5, \dots, p_8$. In other words, to different lines $\ell_i \neq \ell_j$ corresponds different singular points $p_i \neq p_j$.

\item It is easily seen that the line joining $p_i$ and $O=(0,0)$ is invariant. So we can assume that this line is $\ell_{i-4}$.
 
\item For every $j \in \{1,2,3,4\}$ define the polynomial $P_j(t)=B(1,t)-m_j A(1,t)$, where $m_j$ is the slope of $\ell_j$ (which is equal to the slope of $X \mid_{\ell_{j+4}}$). 
Since the roots of the polynomial $P_{j}$ are the slopes of the lines (all passing through $O$) in $\mathrm{Tang}(\FF,\mathcal R_{p_{j+4}})$, it
%
has a double root at $n_j =$ slope of $\ell_{j+4}$ and consequently we have $\mathrm{Disc}(P_j(t),t)=0$. On the other hand, since $\deg P_{j}=3$ and $m_{j}\neq n_{j}$ we have $P_j ' (m_j) \neq 0$.

\end{enumerate}

After conjugating with an automorphism of $\PP$ we may assume $yA-xB= xy(y-x)(y - \nu x)$, that is, the $\FF$-invariant lines through $O$ are $\ell_1=\{x=0\}$, $\ell_2=\{y=0\}$, $\ell_3=\{y=x\}$ and $\ell_4=\{y=\nu x\}$, with $\nu \neq 0,1, \infty$. In this case, the foliation $\FF$ is induced by a vector field of the form
$$
X=\left[(\nu + \alpha_1)x^3 + (\alpha_2 - \nu -1)x^2 y +(\alpha_3 +1) x y^2 \right]\Dx + \left[ \alpha_1 x^2 y + \alpha_2 x y^2 + \alpha_3 y^3 \right]\Dy,
$$
for some complex numbers $\alpha_1$, $\alpha_2$ and $\alpha_3$.

Solving the system $\mathrm{Disc}(P_i(t),t)=0$, $t=1,2,3,4$ by using MAPLE, we arrive to a characterization of the homogeneous foliations of degree $3$ with reduced inflection divisor and flat dual web.

\begin{proposition}\label{homogeneo_generico}
Let $\FF$ be a homogeneous foliation of degree $3$ such that $I(\FF)$ is reduced and $\Leg \FF$ is flat. Then, up to conjugation with an automorphism of $\PP$, the foliation $\FF$ is defined by the vector field
$$\left(\frac{1}{4}\nu x^3-\left( \frac{1}{2}+\frac{1}{2}\nu\right)x^2y+\frac{3}{4}xy^2 \right)\frac{\partial}{\partial x}+
 \left(-\frac{3}{4} \nu x^2y+\left(\frac{1}{2}+\frac{1}{2}\nu \right)xy^{2}-\frac{1}{4}y^3 \right)\frac{\partial}{\partial y} ,$$
with $\nu=\frac{1}{2}\pm\frac{\sqrt{3}i}{2}$. 
\end{proposition}
\section{Main theorem}
In this section we prove our main result: a generalization of proposition \ref{homogeneo_flat} for any degree $3$ foliation. 
We begin with some auxiliary results.

\begin{lemma}\label{eta} Let $h_{i}(x,y)$ be germs of holomorphic functions at the origin of $\CC^{2}$.
Consider $\omega_{i}=dy+y^{a}h_{i}(x,y)dx$, $i=1,2,3$, with $a\in\mathbb N$ 
such that $y$ does not divide $h_{ij}(x,y):=h_{i}(x,y)-h_{j}(x,y)$ for $i\neq j$. Then the curvature of the $3$-web given by $\prod\limits_{i=1}^{3}\omega_{i}=0$ is holomorphic along $y=0$ if and only if $y^{a}$ divides $$\frac{\partial}{\partial x}\left(\frac{h_{12}(x,y)\partial_{x}h_{23}(x,y)-h_{23}(x,y)\partial_{x}h_{12}(x,y)}{h_{12}(x,y)h_{23}(x,y)h_{31}(x,y)}\right).$$ 
\end{lemma}

\begin{proof}
We mimic the proof of Lemma 2.2 in \cite{MarPer}. 
Writing $\omega_{i}\wedge\omega_{j}=\delta_{ij}dx\wedge dy$ we obtain a normalization $\omega_{1}\delta_{23}+\omega_{2}\delta_{31}+\omega_{3}\delta_{12}=0$. Then, there exists a unique $1$-form $\eta=U\,dx+V\,dy$ such that for each cyclic permutation $(i,j,k)$ of $(1,2,3)$ we have $d(\delta_{ij}\omega_{k})=\delta_{ij}\omega_{k}\wedge\eta$. Using that $\delta_{ij}=y^{a}h_{ij}$ we obtain the following linear system satisfied by the unknowns $U$ and $V$:
\begin{eqnarray*}
h_{12}U-h_{12}h_{3}y^{a}V&=&\partial_{x}h_{12}-2ay^{a-1}h_{12}h_{3}-y^{a}\partial_{y}(h_{12}h_{3})\\
h_{23}U-h_{23}h_{1}y^{a}V&=&\partial_{x}h_{23}-2ay^{a-1}h_{23}h_{1}-y^{a}\partial_{y}(h_{23}h_{1})
\end{eqnarray*}
whose unique solution satisfies that $U(x,y)$  and 
$$V(x,y)-\left(\frac{h_{12}(x,y)\partial_{x}h_{23}(x,y)-h_{23}(x,y)\partial_{x}h_{12}(x,y)}{h_{12}(x,y)h_{23}(x,y)h_{31}(x,y)}\right)\frac{1}{y^{a}}-\frac{2a}{y}$$
are holomorphic along $y=0$ thanks to $y\nmid h_{12}h_{23}h_{31}$. The result follows now easily from the fact that the curvature of the given $3$-web is  $d\eta=(\partial_{x}V-\partial_{y}U)dx\wedge dy$.
\end{proof}

\begin{lemma}\label{homogeneo_flat_si_y_sólo_si_no_homogeneo_flat}
Let $\FF$ be a foliation induced by a vector field of the form $X = X_3 +H\cdotp R$, with $X_{3}$ not having multiple components on his singular set,
around the singularity $O=(0,0)$ and denote by $\FF_h$ the homogeneous foliation defined by $X_3$. Then $K(\Leg \FF)$ is holomorphic at the generic point of $\check{O}$ if and only if $K(\Leg \FF_h)$ is holomorphic at the generic point of $\check{O}$. 
\end{lemma}

\begin{proof} The web
$\Leg(\FF)$ is given by the implicit differential equation $F(p,q;x):=F_{3}(p,q;x)+q\,G_{3}(p,q;x)=0$ where $x=-\frac{dq}{dp}$ and $F_{3}(p,q;x)=0$ is the implicit differential equation defining $\Leg(\FF_{h})$. From the homogeneity of $X_{3}$ follows that $F_{3}(p,q;q\lambda)=q^{3}\tilde F_{3}(p,\lambda)$. From the square-free property on $X_{3}$ follows that the discriminant of $\tilde F_{3}(p,\lambda)$ with respect to $\lambda$ does not vanish identically. Hence, we can write $\tilde F_{3}(p,\lambda)=c(p)\prod\limits_{i=1}^{3}(\lambda-\lambda_{i}(p))$, with $\lambda_{i}(p)\not\equiv\lambda_{j}(p)$ if $i\neq j$. Consequently, 
$$F_{3}(p,q;x)=c(p)\prod\limits_{i=1}^{3}(x-q\lambda_{i}(p))\quad\textrm{and}\quad F(p,q;x)=c(p,q)\prod\limits_{i=1}^{3}(x-qh_{i}(p,q)),$$
with $h_{i}(p,0)=\lambda_{i}(p)$. Then we can apply Lemma~\ref{eta} with $a=1$ in order to deduce that $K(\Leg\FF)$ is holomorphic along $\check{O}=\{q=0\}$ if and only if
$$\frac{\partial}{\partial p}\left(\frac{h_{12}(p,0)\partial_{p}h_{23}(p,0)-h_{23}(p,0)\partial_{p}h_{12}(p,0)}{h_{12}(p,0)h_{23}(p,0)h_{31}(p,0)}\right)=0,$$
which is exactly the condition for the holomorphy of $K(\Leg(\FF_{h}))$ at the generic point of $\check{O}$ by applying again Lemma~\ref{eta}.
\end{proof}

\begin{lemma}\label{no_homogeneo_holomorfo_implica_homogeneo_holomorfo}
Under the above notation, if moreover $\Sing(\FF_h)$ has no multiple components and $K(\Leg \FF)$ is holomorphic at the generic point of $\GG_{\FF}(I(\FF))$ then $K(\Leg \FF_h)$ is holomorphic at the generic point of $\GG_{\FF_h}(I(\FF_h))$.  
\end{lemma}

\begin{proof}
For any $c \in \CC$ consider $h_c(x,y)=(cx,cy)$ and $\FF_{c}= h_{c}^{*}(\FF)$ the foliation given by $X_3+c\cdotp H\cdotp R$. Combining our hypothesis with proposition~\ref{holomorfia_singularidad_menor_2} and lemma~\ref{solo_una_singularidad_3} we deduce that $K(\Leg \FF_c)$ is holomorphic on $\Pd\setminus \check{O}$ whenever $c\neq 0$.
If every component of $I(\FF_h)$ is $\FF_h$-invariant
the image by $\mathcal{G}_{\FF_{h}}$	of an irreducible component of $I(\FF_{h})$ is a point. Being holomorphic in a pointed neighborhood of this point, the curvature 
$K(\Leg\FF)$ extends holomorphically by Hartogs extension theorem.
Now take a non invariant line $\ell$ of $I(\FF_h)$. Since $X_3$ is homogeneous, the tangent directions of $\FF_{h}\mid_{\ell}$ are all parallel, thus $\GG_{\FF_h}(\ell)$ is a line of the form $\{p = p_0\}$ in local coordinates $(p,q)$ on $\Pd$ (even if $\ell$ is a line of singularities). Take a generic point $B \in \{p = p_0\}$ and a transverse 
holomorphic section $\Sigma$ through $B$ parametrized by the unit 
 disc $\DD$. Taking the coefficients of the forms $K(\Leg \FF_c)\mid_{\Sigma}$ we obtain a family $f_c(z)$ of functions such that $f_c$ is holomorphic in $\DD$ if $c\neq 0$, and $f_{0}$ is holomorphic in $\DD^{*}$. By Cauchy's theorem one can write
$$
f_c(z)=\frac{1}{2 \pi i}\int_{\mid w \mid = \epsilon}\frac{f_c(w)}{z-w}dw,
$$
for $\epsilon >0$ small enough, $\mid z \mid< \epsilon /2$ and $0 < \mid c \mid \leq 1$. We observe that the right side is a uniformly bounded family even considering $c=0$ and we use Montel's theorem to deduce that $f_{0}$ is also holomorphic at $0$. This
finishes the proof. 
\end{proof}

In order to state our main result we need the following lemma.

\begin{lemma}\label{forma_con_polos_q}
Given a meromorphic $2$-form $\omega = \frac{P(p,q)}{Q(p,q)}dp \wedge dq$ on $\Pd$. Suppose that $q \nmid P$ and $(\omega)_{\infty} \subseteq \check{O} = \{q=0\}$. Then $q^3 \mid Q$.
\end{lemma}

\begin{proof}
By hypothesis we can write $\omega = \frac{T(p,q)}{q^{\nu}}dp \wedge dq$. After doing change of coordinates $a= \frac{p}{q}$ and $b=\frac{1}{q}$ we have $\omega = b^{\nu - n -3} \tilde{T}(a,b) da \wedge db$, where $n=\deg(T)$. Since $\omega$ has no poles at the infinite line we must have $\nu \geq n+3 \geq 3$.
\end{proof}

Now we state our main result.

\begin{theorem}\label{teorema-principal}
Let $\FF$ be a foliation of degree $3$ on $\PP$ such that $\Sing(\FF)$ has no multiple components. Then $\Leg \FF$ is flat if and only if $K(\Leg \FF)$ is holomorphic at the generic point of $\GG_{\FF}(I(\FF))$.
\end{theorem}

\begin{proof}
By proposition \ref{holomorfia_singularidad_menor_2} and lemma \ref{solo_una_singularidad_3} we just need to look at a singular point, which we may take as being $O=(0,0)$, of the form $X= X_3 + H\cdotp R$ as before. 

If $\Sing(\FF_h)$ has no multiple components then we invoke lemma \ref{no_homogeneo_holomorfo_implica_homogeneo_holomorfo} and proposition \ref{homogeneo_flat} to conclude that $K(\Leg \FF_h) \equiv 0$. It follows from lemma \ref{homogeneo_flat_si_y_sólo_si_no_homogeneo_flat} that $K(\Leg \FF)$ is holomorphic on $\PP$ and so it must be zero. 

If $\Sing(\FF_h)$ has a multiple component, without loss of generality we can write 
$$
X_3 = y^2 \cdotp \left[(a_1 x+ b_1 y)\Dx  + (a_2 x+ b_2 y)\Dy \right] , \hspace{0.1cm} H(x,y)= c_0 y^3 + c_1 xy^2 + c_2 x^2 y + c_3x^3.
$$
Using Ripoll's maple script (see \cite{Rip}) we obtain that $K(\Leg \FF)$ has the form
$$
K(\Leg \FF)=\frac{P_{0}(p,q)}{q\,R_{0}(p,q)^{2}}\,dp\wedge dq
$$
with $R_{0}(p,0)=4c_{3}(a_{1}p-a_{2})^{3}$. 
Reasoning by reductio ad absurdum,
assume that $K(\Leg\FF)\not\equiv 0$ is holomorphic at the generic point of $\GG_{\FF}(I(\FF))$. 
By lemma~\ref{discriminant} and proposition~\ref{holomorfia_singularidad_menor_2}
we deduce that $K(\Leg\FF)$ is holomorphic on $\check{\PP}\setminus\check{O}$.
Using lemma~\ref{forma_con_polos_q} we obtain that $q|R_{0}$ and we can distinguish two situations:
\begin{enumerate}[(i)]
\item If $c_3 \neq 0$ then $a_1 = a_2 = 0$ and maple give us 
$$
K(\Leg \FF)=\frac{P_{1}(p,q)}{q\,R_{1}(p,q)^{2}}\,dp\wedge dq
$$
with $R_{1}(p,0)=4c_{3}(b_{1}p-b_{2})^{3}$. Thus $b_1=b_2=0$ and $X_3=0$. 
\item If $c_3=0$ then  Maple give us the equality
$$K(\Leg \FF)=\frac{P_{2}(p,q)}{q\,R_{2}(p,q)^{2}}\,dp\wedge dq$$ with $R_{2}(p,0)=(a_{1}p-a_{2})^{2}$ and necessarily $a_{1}=a_{2}=0$. Again by maple we have
$$K(\Leg \FF)=\frac{P_{3}(p,q)}{q\,R_{3}(p,q)^{2}}\,dp\wedge dq$$ with $R_{3}(p,0)=4c_{2}(b_{1}p-b_{2})$. Since $\Sing(\FF)$ has no multiple components $c_{2}\neq 0$. In consequence $b_{1}=b_{2}=0$ and therefore $X_{3}=0$.
\end{enumerate}
In both cases we obtain that $X_3=0$ so that $X=H\cdot R$ and consequently
 $\Leg\FF$ is an algebraic web hence flat. Then we arrive to contradiction with the assumption $K(\Leg\FF)\not\equiv 0$.
\end{proof}

If $C$ is a  non invariant irreducible  component of the inflection divisor of a degree $3$ foliation $\FF$ on $\PP$ then we consider the curve $C^{\perp}$ consisting of those points $q$ for which there exists  $p\in C$ such that $\Tang(\FF,T_{p}\FF)=2p+q$.
In \cite[Proposition~3.5]{MarPer} it is stated that if $I(\FF)$ is reduced then a necessary condition for $\Leg\FF$ being flat is that for each non invariant irreducible component $C$ of $I(\FF)$, the curve $C^{\perp}$ is invariant by $\FF$.
Let  $I^{\mathrm{tr}}(\FF)$ be the transverse inflection divisor of $\FF$ obtained by removing from $I(\FF)$ the equations of the $\FF$-invariant lines.
As a consequence of Theorem~\ref{teorema-principal} we obtain the following characterization of the flatness of the dual web of a saturated degree $3$ foliation on $\PP$ which extends Proposition~3.5 of \cite{MarPer}.

\begin{corollary}
Let $\FF$ be a saturated degree $3$ foliation on $\PP$ with reduced transverse inflection divisor $I^{\mathrm{tr}}(\FF)$. A necessary and sufficient condition for $\Leg\FF$ being flat is that for each non invariant irreducible component $C$ of $I^{\mathrm{tr}}(\FF)$ we have that $C^{\perp}$ is invariant by $\FF$.
\end{corollary}

Another application of Theorem~\ref{teorema-principal} is the following generalization of Theorem~4.2 in \cite{MarPer} for the degree $3$ case, for which the hypothesis \emph{convex reduced foliation} has been weakened to \emph{saturated convex foliation}, i.e. we do not have to assume any more that the inflection divisor is reduced.

\begin{corollary}
If $\FF$ is a saturated convex foliation of degree $3$ on $\PP$, then $\Leg(\FF)$ is flat.
\end{corollary}

\section*{Appendix: Ripoll's Maple script}

For reader's convenience we include here the maple script of O. Ripoll \cite{Rip} computing the curvature (\verb|kw(F)|) of a $3$-web presented by an implicit differential equation $F(x,y,p)=0$ with $p=\frac{dy}{dx}$.

\bigskip

\footnotesize
\begin{verbatim}
with(LinearAlgebra):kw:=proc(F) local a0,a1,a2,a3,R,alpha0,alpha1,alpha2,k: 
a3:=coeff(F,p,0):a2:=coeff(F,p,1):a1:=coeff(F,p,2):a0:=coeff(F,p,3):R:=Determinant(
Matrix([[a0,a1,a2,a3,0],[0,a0,a1,a2,a3],[3*a0,2*a1,a2,0,0],[0,3*a0,2*a1,a2,0],
[0,0,3*a0,2*a1,a2]])); alpha0:=[diff(a0,y),diff(a0,x)+diff(a1,y),diff(a1,x)+diff(a2,y),
diff(a2,x)+diff(a3,y),diff(a3,x)];alpha1:=Determinant(Matrix([alpha0,[a0,a1,a2,a3,0],
[-a0,0,a2,2*a3,0],[0,-2*a0,-a1,0,a3],[0,0,-3*a0,-2*a1,-a2]]));alpha2:=Determinant(
Matrix([alpha0,[0,a0,a1,a2,a3],[-a0,0,a2,2*a3,0],[0,-2*a0,-a1,0,a3],
[0,0,-3*a0,-2*a1,-a2]])); k:=simplify(diff(alpha2/R,x)+diff(alpha1/R,y)): end proc:
K:=proc(V) local A,B,F,k: A:=V[1]:B:=V[2]:
F:=subs([p=x,q=y,x=-p],subs(y=p*x+q,B-p*A)):k:=subs([x=p,y=q],kw(F)): end proc:
\end{verbatim}

\bigskip

\normalsize
Moreover, the command \verb|K([A,B])| computes directly 
 the curvature of the dual web of the a degree $3$ foliation $\FF$ given by the vector field $A\partial_{x}+B\partial_{y}$ in the affine coordinates $(p,q)$ of $\check{\PP}$. For instance, 
$$\verb|K([|x^{3},y^{3}-1\verb|])| ={\frac { \left( 3\,{p}^{4}+22\,{p}^{2}-10\,{q}^{3}{p}^{2}-25+18
\,{q}^{3}+7\,{q}^{6} \right) p{q}^{2}}{-3 \left( {p}^{4}-2\,{q}^{3}{p}^{
2}-2\,{p}^{2}+{q}^{6}+1-2\,{q}^{3} \right) ^{2}}}$$
times $dp\wedge dq$ is the curvature of $\Leg(x^{3}\partial_{x}+(y^{3}-1)\partial_{y})$.

\bibliographystyle{plain}
\bibliography{flat}

\begin{thebibliography}{10}

\bibitem{Br}
M.~Brunella.
\newblock {\em Birational geometry of foliations}.
\newblock IMPA, 2000.

\bibitem{Fish}
G.~Fischer.
\newblock {\em Plane Algebraic Curves}, volume 15 of Student Mathematical
  Library.
\newblock American Mathematical Society, 2001.

\bibitem{Henaut}
A.~H\'{e}naut.
\newblock Planar web geometry through abelian relations and singularities.
\newblock {\em Nankai Tracts Math.}, 11:269--295, 2006.

\bibitem{Ince}
E.~Ince.
\newblock {\em Ordinary Differential Equations}.
\newblock Dover Publications, 1944.

\bibitem{MarPer}
D.~Mar\'{i}n and J.V. Pereira.
\newblock Rigid flat webs on the projective plane.
\newblock {\em Asian Journal of Mathematics}, 17:163--192, 2013.

\bibitem{Per}
J.~V. Pereira.
\newblock Vector fields, invariant varieties and linear systems.
\newblock {\em Ann. Inst. Fourier (Grenoble)}, 51(5):1385--1405, 2001.

\bibitem{CDQL}
J.~V. Pereira and L.~Pirio.
\newblock Classification of exceptional CDQL webs on compact complex surfaces.
\newblock {\em IMRN}, 12:2169--2282, 2010.

\bibitem{PiPer}
J.V. Pereira and L.~Pirio.
\newblock {\em An invitation to web geometry}.
\newblock IMPA, 2009.

\bibitem{Rip}
O.~Ripoll.
\newblock {\em G\'{e}om\'{e}trie des tissus du plan et \'{e}quations
  diff\'{e}rentielles}.
\newblock Th\`ese de Doctorat de l'Universit\'e Bordeaux 1, 2005.

\bibitem{Ripoll}
O.~Ripoll.
\newblock Properties of the connection associated with planar webs and
  applications.
\newblock {\em Preprint arXiv:math/0702321v2}, 2007.

\end{thebibliography}











\end{document}